\documentclass[12pt,a4paper]{amsart}

\usepackage{amsfonts, amsmath, amssymb, amsthm, amscd, hyperref}

\usepackage{a4wide}

\newtheorem{thm}{Theorem}
\newtheorem{lem}{Lemma}
\newtheorem{con}{Conjecture}

\theoremstyle{definition}
\newtheorem{defn}{Definition}

\theoremstyle{remark}
\newtheorem*{rem}{Remark}

\newcommand{\Fp}{{\mathbb F_p}}

\DeclareMathOperator{\sign}{sign}

\begin{document}

\title{Partitions of nonzero elements of a finite field into pairs}

\author{R.N.~Karasev}

\thanks{The research of R.N.~Karasev is supported by the Dynasty Foundation,
the President's of Russian Federation grant MK-113.2010.1, the Russian Foundation for Basic Research grants 10-01-00096 and 10-01-00139, the Federal Program ``Scientific and scientific-pedagogical staff of innovative Russia'' 2009--2013}

\email{r\_n\_karasev@mail.ru}
\address{
Roman Karasev, Dept. of Mathematics, Moscow Institute of Physics
and Technology, Institutskiy per. 9, Dolgoprudny, Russia 141700}

\author{F.V.~Petrov}

\email{fedyapetrov@gmail.com}

\address{
Fedor Petrov, Saint-Petersburg Dept. of the Steklov Mathematical Institute, nab. Fontanki 27, Saint-Petersburg, Russia 191023}

\thanks{The research of F.V.~Petrov is supported by the Russian Foundation for Basic Research grant 08-01-00379}

\keywords{Finite fields, Combinatorial Nullstellensatz, the Borsuk--Ulam theorem}
\subjclass[2000]{05B40, 05E15, 57S17}

\begin{abstract}
In this paper we prove that the nonzero elements of a finite field with odd characteristic can be partitioned into pairs with prescribed difference (maybe, with some alternatives) in each pair. The algebraic and topological approaches to such problems are considered. We also give some generalizations of these results to packing translates in a finite or infinite field, and give a short proof of a particular case of the Eliahou--Kervaire--Plaigne theorem about sum-sets.
\end{abstract}

\maketitle

\section{Introduction}

In this paper we prove several theorems on combinatorics of finite fields. Denote $\Fp$ the finite field of size $p$, where $p$ is a prime. As it was shown in~\cite{pm2009}, if $p$ is an odd prime, the partitioning $\Fp^*$ into pairs with strictly prescribed differences is possible. Let us denote $[n] = \{1,2,\ldots, n\}$ and give the formal statement.

\begin{thm}
\label{p-part}
Let $p$ be an odd prime, $m=\dfrac{p-1}{2}$. Suppose we are given $m$ elements $d_1, d_2,\ldots, d_m\in \Fp^*$. Then there exist pairwise distinct $x_1,\ldots,x_m,y_1,\ldots,y_m\in \Fp^*$ such that for every $i=1,\ldots, m$ we have
$$
y_i-x_i  = d_i.
$$
\end{thm}

In this paper we present new proofs of this theorem, using algebraic and topological techniques. We also prove some generalizations of Theorem~\ref{p-part}, for example the following result on packing translates in a field.

\begin{defn}
Let $\mathbb F$ be some field. For $X\subseteq \mathbb F$ and $t\in \mathbb F$ denote
$$
X+t = \{x+t: x\in X\},
$$
and for $X\subseteq \mathbb F$ and $Y\subseteq F$ denote
$$
X\pm Y = \{x\pm y : x\in X,\ y\in Y\}.
$$
\end{defn}

\begin{thm}
\label{f-tr-pack}
Suppose $\mathbb F$ is a field, $m$ and $d$ are positive integers such that $\frac{(md)!}{(d!)^m}\ne 0$ in $\mathbb F$. Let $X_1, \ldots, X_m$ and $T_1,\ldots, T_m$ be subsets of $\mathbb F$ such that
$$
\forall i<j\ |X_i-X_j| \le 2d,\quad\forall i\ |T_i|\ge (m-1)d+1.
$$
Then there exists a system of representatives $t_i\in T_i$ such that the sets
$$
X_1+t_1,\ldots, X_m+t_m
$$
are pairwise disjoint.
\end{thm}

In particular, if $\mathbb F = \Fp$, $T_i=\Fp$, and $md<p$, then we can translate the sets $X_1, \ldots, X_m\subset\Fp$ so that they become pairwise disjoint, provided $|X_j-X_i|\le 2d$ for all $i<j$. Theorem~\ref{p-part} is a particular case of Theorem~\ref{f-tr-pack} with $\mathcal F = \mathcal F_p, d=2, X_i=\{0, d_i\}$, and $T_i=\mathbb F\setminus \{0, -d_i\}$.

Let us return to partitions into pairs of other finite Abelian groups. For finite fields of size $p^k$ (we treat them as $\Fp$-vector spaces) the differences cannot be prescribed strictly. The simple counterexample is when the difference is the vector $d=(1,0,0,\ldots,0)$ for all pairs, then every line $\{v+td\}_{t\in \Fp}$ can have at most $\lfloor\frac{p}{2}\rfloor$ pairs and the partition is impossible.

The same obstruction arises if we try to generalize Theorem~\ref{p-part} for the rings $\mathbb Z/(n)$ for odd composite $n$. Here we may require all $d_i$ to be $d=n/p$, where $p$ is prime divisor of $n$. It is clear that in every set $\{v+dt\}_{t\in \Fp}$ one element would be not paired. Further conjectures for partitions of $\mathbb Z/(n)$ are discussed in Section~\ref{zn-pairs-sec}.

Now return to the positive results. In the case of the finite field of size $p^k$ (treated as $\Fp$-vector space here), it is sufficient to give some alternatives for each pair, which is done in the following theorem.

\begin{thm}
\label{pk-part}
Let $p$ be an odd prime, and let $V$ be the $\Fp$-vector space of dimension $k$. Denote $V^* = V\setminus\{0\}$ and put $m=|V^*|/2=\dfrac{p^k-1}{2}$. Suppose we are given $m$ linear bases of the vector space $V$
$$
(v_{11}, \ldots, v_{1k}), (v_{21}, \ldots, v_{2k}),\ldots, (v_{m1}, \ldots, v_{mk})
$$
Then there exist pairwise distinct $x_1,\ldots, x_m, y_1,\ldots, y_m \in V^*$ and
a map $g : [m]\to [k]$ such that for every $i=1,\ldots, m$ we have
$$
y_i-x_i = v_{ig(i)}.
$$
\end{thm}

The paper is organized as follows. We discuss different proofs of Theorem~\ref{p-part} in Section~\ref{p-part-sec}. In Section~\ref{dyson-sec} we give a new algebraic proof of a lemma on discriminant-like polynomials, known as the Dyson conjecture~\cite{dyson1962}, which is used in the proofs of Theorems~\ref{p-part} and \ref{f-tr-pack}. In Section~\ref{f-tr-pack-sec} we prove Theorem~\ref{f-tr-pack}. The topological proofs for Theorems~\ref{p-part} and \ref{pk-part} are given in Section~\ref{top-sec}. In Sections~\ref{zn-pairs-sec} and \ref{cau-dav-sec} we discuss some conjectures and similar results for cyclic groups.

Generally, the algebraic methods in combinatorics are known to be very useful and powerful, see~\cite{frwi1981,kaka1993,alon1999} for examples of their application. The topological methods in combinatorics and discrete geometry also proved to be very useful, starting from the lower bounds for the chromatic number of the Kneser graphs in~\cite{lov1978}, other examples of topological methods can be found in~\cite{mat2003,ziv2004}.

The authors thank Noga~Alon, Doron Zeilberger, and Michal Adamaszek for useful discussion and comments. We  thank D\"om\"ot\"or P\'alv\"olgyi for drawing our attention to the known proofs of Theorem~\ref{p-part}; and we thank the unknown referee for numerous remarks, corrections, and references.

\section{The algebraic proofs of Theorem~\ref{p-part}}
\label{p-part-sec}

First, we sketch a simplified version of the proof in~\cite{pm2009}. It is in the spirit of the Combinatorial Nullstellensatz~\cite{alon1999}, see also~\cite{alon2000} for a similar proof of a theorem on distinct pairwise sums in $\Fp$.

For a polynomial $f(x_1,x_2,\dots,x_m)\in \Fp [x_1,x_2,\dots,x_m]$ denote
$$
\int f=\sum_{(c_1,\dots,c_m)\in \Fp^m} f(c_1,c_2,\dots,c_m).
$$

Recall the following lemma, often used to study Diophantine equations over finite fields.

\begin{lem}
\label{average-p}
If $\deg f<m(p-1)$, then $\int f=0$.
\end{lem}

\begin{proof}
This is well-known for $m=1$ (recall the proof: it suffices to consider $f(x)=x^k$, $1\leq k\leq p-2$. There exists $g\in \Fp^*$, which is not a root of polynomial $x^k-1$, for such $g$ we have $S:=\sum_{x\in \Fp} f(x)=\sum_{x\in \Fp} f(gx)=g^k S$, hence $S=0$). In the general case, note that each monomial of degree
less then $m(p-1)$ has degree less then $p-1$ in some specific variable. If we sum up (``integrate'') in this variable first, we get zero by the one-dimensional case.
\end{proof}

We interpret our problem as follows. We need to find elements $c_1,c_2,\dots,c_m$ from $\Fp$ such that elements $c_i$ and $c_i+d_i$ are all distinct and nonzero. That is, it suffices to prove that the following polynomial $f$ takes non-zero values
\begin{multline*}
f(x_1,\dots,x_m)=\\=x_1\dots x_m(x_1+d_1)\dots(x_m+d_m)
\prod_{i<j} (x_i-x_j)\cdot (x_i+d_i-x_j)\cdot (x_i-x_j-d_j)\cdot (x_i+d_i-x_j-d_j).
\end{multline*}

It suffices to prove that $\int f\ne 0$. Note that this polynomial has degree $2m+4m(m-1)/2=2m^2=m(p-1)$, and its homogeneous component of the maximal degree does not depend on $d_1$, $d_2$, $\dots$, $d_m$. By Lemma~\ref{average-p} it means that $\int f=\int g$, where $g$ is any polynomial with the
same component of maximal degree. Put
$$
g(x_1,\dots,x_m)= \prod_{j=1}^m (x_i^2+x_i) \cdot \prod_{i<j} (x_i-x_j)^2\cdot ((x_i-x_j)^2-1).
$$
Note that $g(c_1,c_2,\dots,c_m)\ne 0$ iff $c_1$, $c_2$, $\dots$, $c_m$
are $1$, $3$, $\dots$, $p-2$ (odd numbers) in some order. It is clear that all these $m!$ non-zero
values of $g$ are equal. It follows that $\int g\ne 0$, hence $\int f\ne 0$ and the proof is complete.

Let us discuss another approach to this theorem via the Combinatorial Nullstellensatz, similar to the technique of~\cite{alon2000}, this approach to Theorem~\ref{p-part} was also discovered independently in~\cite{ks2010}.

Recall the Combinatorial Nullstellensatz.

\begin{thm}
\label{cn-prod}
Suppose a polynomial $f(x_1,x_2,\dots,x_n)$ over field $\mathbb{F}$ has degree at most $c_1+c_2+\dots+c_n$, where $c_i$ are non-negative integers, and denote by $C$ the coefficient at $x_1^{c_1}\dots x_n^{c_n}$ in $f$ (maybe, $C=0$). Let $A_1$, $A_2$, \ldots, $A_n$ be arbitrary subsets of $\mathbb{F}$ such that $|A_i|=c_i+1$ for any $i$. Denote also $\varphi_i(x)=\prod_{\alpha\in A_i}(x-\alpha)$. Then
\begin{equation}
\label{cn-eq}
C=\sum_{\alpha_i\in A_i} \frac{f(\alpha_1,\dots,\alpha_n)}
{\varphi_1'(\alpha_1)\dots \varphi_n'(\alpha_n)}
\end{equation}
In particular, if $C\ne 0$, then there exists a system of representatives $\alpha_i\in A_i$ such that $f(\alpha_1,\alpha_2,\dots,\alpha_n)\ne 0$.
\end{thm}

\begin{proof} For $n=1$, (\ref{cn-eq}) just follows from the
Lagrange interpolation formula, which gives the representation
$$
f(x)=\sum_{\alpha\in A_1} f(\alpha)\frac{\varphi_1(x)}{\varphi_1'(\alpha)(x-\alpha)}.
$$
By induction on $n$ (\ref{cn-eq}) also holds for any monomial of degree at most $c_i$ in each $x_i$. Next, by linearity of both parts of (\ref{cn-eq}) it suffices to prove (\ref{cn-eq}) for $h:=f-Cx_1^{c_1}\dots x_n^{c_n}$. Each  monomial of $h$ has degree less then $c_i$ for at least one index $i$. If we fix all values $\alpha_j$ for $j\ne i$, then summation over $\alpha_i\in A_i$ gives 0, as follows again from the one-dimensional case.
\end{proof}

\begin{rem} The formula (\ref{cn-eq}) after multiplying by the common denominator holds also for commutative rings with unity. Indeed, for fixed $c_i$'s both parts are some polynomials with integral coefficients in $x_i$'s, elements of $A_i$'s and coefficients of $f$. Since (\ref{cn-eq}) holds over $\mathbb{R}$, these polynomials should be identically equal.
\end{rem}

Let us use this theorem for the polynomial
$$
f = \prod_{i<j} (x_i-x_j)\cdot (x_i+d_i-x_j)\cdot(x_i-x_j-d_j)\cdot(x_i+d_i-x_j-d_j),
$$
the numbers $c_i=2m-2=p-3$, and $A_i=\Fp^*\setminus\{-d_i\}$. The only thing left to check is that the coefficient of $\prod x_i^{2m-2}$ in $f$ does not vanish. It equals the coefficient of $\prod x_i^{2m-2}$ in $\prod_{i<j} (x_i-x_j)^4$. The Dyson conjecture (\cite{dyson1962}, or \cite[Theorem~3.2]{alon2000}, or section \ref{dyson-sec} here), proved in~\cite{wilson1962,gunson1962}, states (in its particular case) that this coefficient equals $\dfrac{(2m)!}{2^m}$, which is clearly different from $0$ in $\Fp$, so we are done. 

The approach in our first proof of Theorem~\ref{p-part} avoids the Dyson conjecture (and the Combinatorial Nullstellensatz), but actually it provides an alternative proof of it. This proof is given in Section~\ref{dyson-sec}.

\section{A new proof of the Dyson conjecture}
\label{dyson-sec}

The Combinatorial Nullstellensatz is often used for getting information on values of polynomials from the knowledge of their coefficients. But (\ref{cn-eq}) allows to use it in other direction, as we show by deriving the Dyson conjecture.

\begin{thm}
\label{dyson}
Let $a_i$, $1\leq i\leq n$ be positive integers. Denote by $C$ the free term in
$$
\prod_{1\le i\ne j\le n} (1-x_i/x_j)^{a_i}.
$$
In other words, with $a=\sum a_i$, $C$ equals the coefficient of
$\prod x_i^{a-a_i}$ in
\begin{equation}\label{2}
f(x_1,\dots,x_n):=\prod_{1\leq i<j\leq n} (-1)^{a_j}(x_j-x_i)^{a_i+a_j}.
\end{equation}
Then
$$
C=\frac{a!}{a_1!\dots a_n!}.
$$
\end{thm}

\begin{proof} In notations of Theorem~\ref{cn-prod}, we have $c_i=a-a_i$. The idea is to add terms of lower degree to $f$, it does not change the coefficient $C$, but may significantly change the RHS of (\ref{cn-eq}). Also, we are free to choose $A_i$. Let's try to change $f$ to $\tilde{f}$ and choose $A_i$ so that $\tilde{f}$ takes unique non-zero value on $\prod A_i$. Put $A_i=\{0,1,\dots,a-a_i\}$. So,
if $x_i\in A_i$, then the segment $\Delta_i:=[x_i,x_i+a_i-1]$ lies inside $[0,a-1]$. Here and in the rest of the text we denote by $[\ldots]$ segments of integers.

Now we change $f$. Replace $(x_j-x_i)^{a_i+a_j}$ in formula (\ref{2}) for $f$ by
$$
C_{i,j}(x_1,\dots,x_n):=\prod_{s=-a_i+1}^{a_j} (x_j-x_i+s).
$$
Non-vanishing of $C_{i,j}$ means that the segments $\Delta_i$, $\Delta_j$ are disjoint
and $\Delta_i$ may not be the segment following $\Delta_j$
(that is, $\min \Delta_i\ne \max \Delta_j+1$).
All this together may happen only if $\Delta_1$, $\Delta_2$, $\dots$, $\Delta_n$
are consecutive segments $[0,a_1-1]$, $[a_1,a_1+a_2-1]$, $\dots$, $[a-a_n,a-1]$,
i.e. if $x_i=\beta_i:=a_1+\dots+a_{i-1}$.

So, $C$ equals
$$C=\frac{\prod_{1\leq i<j\leq n} (-1)^{a_j} C_{i,j}(\beta_1, \beta_2, \dots,\beta_n)}
{\prod \varphi_i'(\beta_i)},$$
where $\varphi_i(x)=\prod_{s=0}^{a-a_i} (x-s)$.
It may be calculated easily by noticing that
$$
\varphi_i'(\beta_i) = (-1)^{a_{i+1}+\dots+a_n} (a_1+\dots+a_{i-1})!(a_{i+1}+\dots+a_n)!
$$
and
$$
C_{i,j}(\beta_1,\dots,\beta_n)=\frac{(a_i+\dots+a_j)!}{(a_{i+1}+\dots+a_{j-1})!}.
$$

Many factorials and powers of $-1$ are canceled, and we get the
desired formula for $C$.

\end{proof}

\section{The proof of Theorem~\ref{f-tr-pack}}
\label{f-tr-pack-sec}

Choose some sets $M_{ij}\supseteq X_j - X_i$ of size exactly $2d$ and consider the polynomial
$$
f(x_1, \ldots, x_m) = \prod_{1\le i<j\le m} \prod_{t\in M_{ij}} (x_i - x_j - t).
$$
If $f$ attains a nonzero value on $T_1\times\dots\times T_m$ then the proof is complete. Note that
$$
\deg f = m(m-1)d,
$$
and its coefficient at $(x_1\dots x_m)^{(m-1)d}$ is the same as the coefficient in
$$
\tilde f(x_1, \ldots, x_m) = \prod_{1\le i<j\le m} (x_i - x_j)^{2d},
$$
which is
$$
(-1)^{d\binom{m}{2}} \frac{(md)!}{(d!)^m}
$$
by the Dyson conjecture. Now we again apply Theorem~\ref{cn-prod}.

\begin{rem}
The statement of Theorem \ref{f-tr-pack} holds also if we replace the inequality $|X_i-X_j| \le 2d$ by the following condition: the polynomial
$$
\prod_{i<j} (x_i-x_j)^{|X_i-X_j|}
$$
has at least one non-zero monomial with degree at most $|T_i|-1$ in $x_i$ for every $i$. Then we can multiply by an appropriate monomial to have nonzero coefficient at $\prod_{i=1}^m x_i^{|T_i|-1}$ and apply the Combinatorial Nullstellensatz. 

In particular, the above reasoning works if $\mathbb F = T_i = \Fp$ and $|X_i-X_j|\leq a_i+a_j$ for some non-negative integer $a_i$'s such that $\sum a_i\leq p-1$ (by Dyson's conjecture again). In particular, $a_i=\left\lceil \frac{|X_i|^2}2\right\rceil $ satisfy $|X_i-X_j|\leq a_i+a_j$, so the statement holds provided that $\sum \left\lceil \frac{|X_i|^2}2\right\rceil <p$.
\end{rem}

\section{The topological proofs}
\label{top-sec}

Now we go to the topological proofs, as usual they use a certain generalization of the Borsuk--Ulam theorem.

The general examples of using the Borsuk--Ulam theorem in combinatorics can be found in~\cite{mat2003,ziv2004}. In particular, the topological proof of Theorem~\ref{p-part} uses the ideas in~\cite{vuziv1993}, where the lower bound on the number of Tverberg partitions is proved. The proof of Theorem~\ref{pk-part} uses the technique from~\cite{hell2007}, where the number of Tverberg partitions was estimated for the case when the number of parts is a prime power.

Let us state the generalized Borsuk--Ulam theorem that we need (see~\cite{vol2000} for example).

\begin{lem}
\label{bu-pk}
Let $G=(\Fp)^k$ be the additive group. Let $X$ and $Y$ be $G$-CW-complexes
with fixed point free action of $G$. Let $X$ be $n$-connected and
$Y$ be $n$-dimensional. Then there cannot exist a continuous
map $f: X\to Y$, commuting with the action of $G$.
\end{lem}

We are going to use this lemma in the case, when $X$ and $Y$ are simplicial complexes, the action of $G$ and the map $f$ are simplicial. In this case the spaces are indeed $G$-CW-complexes. Such a point of view allows to state everything purely in combinatorial terms, without appealing to topological spaces.

Let us prove Theorem~\ref{p-part}. Consider the following simplicial complex ($*$ means join)
$$
K = V*S_1*\dots*S_m,
$$
where $V$ is a discrete set equal to $\Fp$, $S_i$ is a one-dimensional subcomplex of $V*V$, with edges of type $(x, x)$ and $(x, x+d_i)$. Clearly $K$ is a join of a discrete set and $m$ circles (equivalently, a join of $V$ with $p-2$-dimensional sphere), and therefore $p-2$-connected and $p-1$-dimensional.

Consider the complex $L$, having the same vertices as $V$, and all subsets of $\le p-1$ elements as simplices. The map $f:K\to L$ is defined naturally on vertices. Note the important thing: this map would be simplicial if the required permutation does not exist. Indeed, if a simplex in $K$ is mapped to a non-simplex in $L$, then it is mapped onto the entire set $V$. Hence, it is $p-1$-dimensional of the form
$$
v*[a_1, b_1]*\dots*[a_m, b_m]\in V*S_1*\dots*S_m,
$$
and the set $\{v, a_1, b_1, \ldots, a_m, b_m\}$ equals $V$. Shifting by $-v$ we have $v=0$, and from bijectivity $b_i=a_i+d_i$ (the case $a_i=b_i$ is impossible). Hence we may assume the contrary: $f$ is simplicial.

Note that the map $f$ is $\Fp$-equivariant. Here we identify $V=\Fp$ and consider the action of $V$ on itself by shifts and on $S_i$ by shifting both coordinates by the same value. Thus arises a free action on $K$ (by shifting all the coordinates by the same value), and a free action on $L$ by shifts. Besides, $f$ maps a $p-2$-connected complex to a $p-2$-dimensional complex. Hence such a map cannot exist by Lemma~\ref{bu-pk}, and the required permutation must exist. Theorem~\ref{p-part} is proved.

Now let us prove Theorem~\ref{pk-part}. Consider $V=(\Fp)^k$ and action of $V$ on itself by shifts. Let
$$
K = V*S_1*\dots*S_m,
$$
where the complex $S_i$ has vertices $V*V$ and edges $(x,x)$, $(x, x+ v_{ij})$ for all possible $x\in V, j=1,\ldots,k$. It is essential that $S_i$ is connected iff $\{v_{ij}\}_{j=1}^k$ linearly span $V$, which is required in the theorem. In this case $K$ is also $p^k-2$-connected and $p^k-1$-dimensional. The complex $L$ on vertices $V$ is defined the same way, its simplices are all subsets of size at most $p^k-1$, hence it is $p^k-2$-dimensional. 

Consider the action of $V$ on itself by shifts, and on $S_i$ by shifting both coordinates by the same value. Thus arises a free action on $K$ (by shifting all the coordinates by the same value). The action by shifts on $L$ is \emph{only} fixed point free (not \emph{free}) this time. In this case Lemma~\ref{bu-pk} is essentially needed, while for Theorem~\ref{p-part} we only need its simple particular case for free actions (the Dold theorem, see~\cite{mat2003} for example). The rest of the proof is the same, applying Lemma~\ref{bu-pk} for the group of shifts $V=(\Fp)^k$.

\section{Conjectures on partitions of $\mathbb Z/(n)$}
\label{zn-pairs-sec}

We conjecture the following generalizations of Theorem~\ref{p-part} for rings $\mathbb Z/(n)$. We denote by $\mathbb Z/(n)^*$ the invertible (coprime with $n$) elements of $\mathbb Z/(n)$.

\begin{con}
\label{odd-part}
Let $n=2m+1$ be a positive integer. Suppose we are given $m$ elements $d_1, d_2,\ldots, d_m\in \mathbb Z/(n)^*$. Then there exists a partition of $\mathbb Z/(n)\setminus\{0\}$ into pairs with differences $d_1, d_2,\ldots, d_m$.
\end{con}

The following conjecture was proposed (and verified for $n=24$) by Michal Adamaszek (private communication).

\begin{con}
\label{even-part}
Let $n=2m$ be a positive integer. Suppose we are given $m$ elements $d_1, d_2,\ldots, d_m\in \mathbb Z/(n)^*$. Then there exists a partition of $\mathbb Z/(n)$ into pairs with differences $d_1, d_2,\ldots, d_m$.
\end{con}

Let us discuss the possible approach to these conjectures using the Combinatorial Nullstellensatz, based on ideas from~\cite{dkss2001}. Let us embed $\mathbb Z/(n)$ into $\mathbb C$ as the $n$-th roots of unity (denote these roots by $C_n$). Denote 
$$
w = \cos \frac{2\pi}{n} + i\sin \frac{2\pi}{n}.
$$
The numbers $d_i$ are transformed into $w_i = w^{d_i}$. Consider the polynomial
$$
F(x_1,\dots,x_n)=\prod_{1\leq i< j\leq m} (x_i-x_j)(w_ix_i-x_j)(x_i-w_jx_j)(w_ix_i-w_jx_j).
$$
We have to prove that it takes a nonzero value on $C_n\times\dots\times C_n$. This would be guaranteed by a nonzero coefficient at $\prod x_i^{2m-2}$. Consider 
$$
G = F \prod_i (x_i-w_ix_i)=F \prod_i (1-w_i)\cdot \prod_i x_i.
$$
The polynomial $G$ is the Vandermonde polynomial of the following $2m$ variables
$$
(z_1,\ldots,z_{2m}) = (x_1, x_2, \ldots, x_m, w_1x_1, \ldots, w_mx_m).
$$
We are interested in its coefficient at $\prod x_i^{2m-1}$. The Vandermonde polynomial is the determinant
$$
\sum_\sigma \sign \sigma\  z_{1}^{\sigma_1} z_{2}^{\sigma_2}\dots z_{2m}^{\sigma_{2m}},
$$
where summation is over all bijections $\sigma : [1,2m]\to [0,2m-1]$. A summand is proportional to $\prod x_i^{2m-1}$ iff $\sigma_i+\sigma_{m+i}=2m-1$ for all $1\leq i\leq m$. The parity of such a permutation is determined by the number of $i\in [1,m]$ such that $\sigma_i\geq m$. Therefore the coefficient in $G$ at $\prod x_i^{2m-1}$ equals
\begin{equation}
\label{permanent}
\sum_\pi (w_1^{\pi_1}-w_1^{2m-1-\pi_1})\cdot \dots \cdot (w_m^{\pi_m}-w_m^{2m-1-\pi_m}),
\end{equation}
the summation is over bijections $\pi : [1,m]\to [0,m-1]$. The coefficient at $\prod x_i^{2m-2}$ in $F$ is obtained from this expression dividing by $\prod(1-w_i)$.

In the general case the authors cannot prove that this coefficient is nonzero, but in case $n=p$ is a prime we obtain an alternative proof of Theorem~\ref{p-part} as follows. Let us divide (\ref{permanent}) by $\prod_{i=1}^n (1 - w_i)$ to obtain 
\begin{equation}
\label{permanent2}
\sum_\pi (w_1^{\pi_1}+\dots + w_1^{2m-2-\pi_1})\cdot \dots \cdot (w_m^{\pi_m}+\dots+w_m^{2m-2-\pi_m}),
\end{equation}
and apply the following lemma.

\begin{lem}
Let $f$ be a polynomial with integer coefficients, $p$ be a prime, $w$ be the $p$-th root of unity. If $f(w)=0$, then $f(1)$ is divisible by $p$.
\end{lem}

\begin{proof}
The minimal polynomial of $w$ is $g = 1+w+...+w^{p-1}$. Hence $f=hg$, where $h$ is a polynomial with integer coefficients, and therefore $f(1) = h(1) g(1) = h(1)p$.
\end{proof}

Now we note that after replacing all $w_i$ (by definition they are $w^{d_i}$) in (\ref{permanent2}) by $1$, we obtain 
$$
\sum_\pi (2m-2\pi_1 - 1)\cdot \dots \cdot (2m-2\pi_m-1),
$$
Every summand equals $(2m-1)!!$, and there are $m!$ summands, hence the total value equals $m!(2m-1)!!$ and is not divisible by $p$.

\section{Some remarks on the sum-sets}
\label{cau-dav-sec}

The technique of the previous section allows to give a short proof of a particular case of the Cauchy--Davenport type theorem from~\cite{bole1996,ekp2003} (see also~\cite{cau1813,dav1935,eliker1998,alon1999,kar2004}). The Cauchy--Davenport type theorems estimate the cardinality of 
$$
A+B = \{a + b : a\in A,\ b\in B\},
$$
where $A$ and $B$ are finite subsets of an Abelian group. The technique we are going to use was already used in~\cite{kar2004} in application to the sum-sets problem (and the restricted sum-sets problem).

\begin{defn}
Define $\beta_p(r,s)$ to be smallest integer $n$ such that $p\mid \binom{n}{k}$ for all $k$ in the range
$$
n-r < k < s.
$$
\end{defn}

\begin{thm}
\label{cau-dav-pk}
Let $A, B\subset \mathbb Z/(p^\alpha)$, where $p$ is prime. Then 
$$
|A+B|\ge \beta_p(|A|, |B|).
$$
\end{thm}

\begin{proof}
The reasoning is essentially the same as in~\cite[Theorem~3, Lemma~5]{kar2004}.
Consider $\mathbb Z/(p^\alpha)$ as the multiplicative group of solutions of the equation $z^{p^\alpha}-1=0$ in $\mathbb C$. Put $|A|=r$ and $|B|=s$. By considering $-B$ instead of $B$, we pass to studying the set 
$$
C = \{a/b : a\in A, b\in B\}
$$
and prove that $|C| \ge \beta_p(r, s)$. Assume the contrary, $|C| = n < \beta_p(r, s)$, then the polynomial
$$
f(x,y) = \prod_{c\in C} (x - cy)
$$
is zero on $A\times B$. By the definition of $\beta_p(r, s)$ there exists $n-r < k < s$ such that 
$$
\binom{n}{k} \neq 0\mod p.
$$
Consider the coefficient in $f(x, y)$ at the monomial $x^ky^{n-k}$, which is
$$
d = (-1)^{n-k} \sigma_{n-k}(c_1, \ldots, c_n),
$$
where $\sigma$ is the elementary symmetric function, $C=\{c_1,\ldots, c_n\}$. By the Combinatorial Nullstellensatz (Theorem~\ref{cn-prod}) this coefficient $d$ should be zero.

Let us write every $c_i$ as a power of $z$, the primitive $p^\alpha$-th root of unity. Then $d$ becomes a polynomial of $z$ with integer coefficients, denote it by $d(z)$. Note that $d(z)=0$, and therefore $d(z)$ is divisible by the minimal polynomial of $z$
$$
r(z) = \frac{z^{p^\alpha}-1}{z^{p^{\alpha-1}}-1} = \sum_{i=0}^{p-1} z^{ip^{\alpha-1}}.
$$ 
Substituting into the equality for polynomials with integer coefficients
$$
d(z) = r(z)q(z)
$$
the value $z=1$, we obtain
$$
(-1)^{n-k}\binom{n}{k} = (-1)^{n-k} \sigma_b(1, \ldots, 1) = r(1)q(1) \equiv 0 \mod p,
$$
which is a contradiction with $\binom{n}{k}\neq 0\mod p$.
\end{proof}

\begin{rem}
The Combinatorial Nullstellensatz can also prove this theorem (with the same estimate $\beta_p(\cdot, \cdot)$) for groups of type $(\mathbb Z_p)^k$. These groups are additive groups of fields and the proof is even simpler compared to the above reasoning, the relevant coefficients of $\prod_{c\in C} (x + y - c)$ are already \emph{equal} to $\binom{n}{k}$. A much stronger result is proved in~\cite{bole1996,ekp2003}: Theorem~\ref{cau-dav-pk} holds for any finite Abelian $p$-group.
\end{rem}


\begin{thebibliography}{99}

\bibitem{alon1999}
N.~Alon. Combinatorial Nullstellensatz. // Combin. Probab. Comput., 8, 1999, 7--29.

\bibitem{alon2000}
N.~Alon.  Additive Latin transversals. // Israel J. Math., 117, 2000, 125--130.

\bibitem{bole1996}
B.~Bollob\'as, I.~Leader. Sums in the grid. // Discrete Mathematics, 162, 1996, 31--84.

\bibitem{br1887}
H.~Brunn. \"Uber Ovale und Eifl\"achen. Inaugural Dissertation, M\"unchen, 1887.

\bibitem{cau1813}
A.L.~Cauchy. Recherches sur les nombres. // J. \'Ecole polytech., 9, 1813, 99--116.

\bibitem{dkss2001}
S.~Dasgupta, G.~K\'arolyi, O. Serra, B.~Szegedy. Transversals of additive Latin squares. // Israel Jour. Math., 126(1), 2001, 17--28.

\bibitem{dav1935}
H.~Davenport. On the addition of residue classes. // J. London Math. Soc., 10, 1935, 30--32.

\bibitem{dyson1962}
F.J.~Dyson. Statistical theory of the energy levels of complex systems. I, II, III. // J. Mathematical Phys., 3, 1962, 140--175.

\bibitem{eliker1998}
S.~Eliahou, M.~Kervaire, Sumsets in vector spaces over finite fields. // J. Number Theory, 71, 1998, 12--39.

\bibitem{ekp2003}
S.~Eliahou, M.~Kervaire, A.~Plaigne. Some extensions of the Cauchy-Davenport theorem. // J. Number Theory, 101, 2003, 338--348.

\bibitem{frwi1981}
P.~Frankl, R.M.~Wilson. Intersection theorems with geometric consequences. Combinatorica, 1(4), 1981, 357--368.

\bibitem{gunson1962}
J.~Gunson. Proof of a conjecture of Dyson in the statistical theory of energy levels. // Journal of Mathematical Physics, 3, 1962, 752--753.

\bibitem{hell2007}
S.~Hell. On the number of Tverberg partitions in the prime power case. // European Journal of Combinatorics, 28(1), 2007, 347--355.

\bibitem{kaka1993}
J.~Kahn, G.~Kalai. A counterexample to Borsuk's conjecture. // Bull. Amer. Math. Soc., 29(1), 1993, 60--62.

\bibitem{kar2004}
G.~K\'arolyi. The Erd\"os--Heilbronn problem in Abelian groups. // 
Israel Journal of Mathematics, 139(1), 2004, 349--359.

\bibitem{ks2010}
D.~Kohen, I.~Sadofschi. A new approach on the seating couples problem. // \href{http://arxiv.org/abs/1006.2571}{arXiv:1006.2571}, 2010.

\bibitem{lov1978}
L.~Lov\'asz. Kneser's conjecture, chromatic numbers, and homotopy. // Journal of Combinatorial Theory, Series A, 25(3), 1978, 319--324.

\bibitem{mat2003}
J.~Matou\v{s}ek. Using the Borsuk--Ulam theorem. Berlin-Heidelberg, Springer Verlag, 2003.

\bibitem{min1896}
H.~Minkowski. Geometrie der Zahlen. Leipzig: Teubner, 1896.

\bibitem{pm2009}
E.~Preissmann, M.~Mischler. Seating couples around the King's table and a new characterization of prime numbers. // American Mathematical Monthly, 116(3), 2009, 268--272.

\bibitem{vol2000}
A.Yu.~Volovikov. On the index of $G$-spaces (In Russian). // Mat. Sbornik, 191(9), 2000, 3--22; translation in Sbornik Math., 191(9), 2000, 1259--1277.

\bibitem{vuziv1993}
A.~Vu\v ci\'c, R.T.~\v Zivaljevi\'c. Note on a conjecture of Sierksma. // Discrete and Computational Geometry, 9(1), 1993, 339--349.

\bibitem{wilson1962}
K.~Wilson. Proof of a conjecture of Dyson. // Journal of Mathematical Physics 3, 1962, 1040--1043.

\bibitem{ziv2004}
R.~\v Zivaljevi\'c. Topological methods. // Handbook of Discrete and Computational Geometry, ed. by J.E.~Goodman, J.~O'Rourke, CRC, Boca Raton, 2004.
\end{thebibliography}
\end{document}